\documentclass[11pt]{amsart}

\usepackage{epsfig, color}

\usepackage{amsmath}
\usepackage{amsthm}
\usepackage{hyperref}
\usepackage[margin=1.1in]{geometry}

\numberwithin{equation}{section}

\newtheorem{prop}{Proposition}
\newtheorem{lemma}[prop]{Lemma}

\newtheorem{thm}[prop]{Theorem}
\newtheorem{cor}[prop]{Corollary}
\newtheorem{conj}[prop]{Conjecture}
\numberwithin{prop}{section}

\theoremstyle{definition}
\newtheorem{defn}[prop]{Definition}

\newtheorem{rmk}[prop]{Remark}

\newcommand{\del}{\partial}
\newcommand{\dt}{\frac{\partial}{\partial t}}
\newcommand{\brs}[1]{\left| #1 \right|}

\renewcommand{\gg}{\gamma}
\newcommand{\gD}{\Delta}
\newcommand{\gd}{\delta}

\newcommand{\gk}{\kappa}
\newcommand{\gl}{\lambda}

\newcommand{\gw}{\omega}
\newcommand{\ga}{\alpha}

\renewcommand{\ge}{\epsilon}
\newcommand{\N}{\nabla}
\newcommand{\FF}{\mathcal F}

\newcommand{\EE}{\mathcal E}

\newcommand{\til}[1]{\widetilde{#1}}
\newcommand{\nm}[2]{\brs{\brs{ #1}}_{#2}}

\renewcommand{\bar}[1]{\overline{#1}}

\DeclareMathOperator{\Rc}{Rc}
\DeclareMathOperator{\Rm}{Rm}
\DeclareMathOperator{\inj}{inj}
\DeclareMathOperator{\tr}{tr}

\DeclareMathOperator{\grad}{grad}
\DeclareMathOperator{\Vol}{Vol}
\DeclareMathOperator{\diam}{diam}

\DeclareMathOperator{\Area}{Area}

\DeclareMathOperator{\supp}{supp}
\DeclareMathOperator{\Length}{Length}

\begin{document}

\title[Collapsing in the $L^2$ curvature flow]{Collapsing in the $L^2$ curvature
flow}
\author{Jeffrey Streets}
\address{Rowland Hall\\
         University of California\\
         Irvine, CA 92617}
\email{\href{mailto:jstreets@uci.edu}{jstreets@uci.edu}}

\thanks{The author was partly supported by a grant from the National Science
Foundation}

\begin{abstract} We show some results for the $L^2$
curvature flow linked by the theme of addressing collapsing phenomena.
First we show long time existence and convergence of the flow for
$SO(3)$-invariant initial data on $S^3$, as well as a long time
existence and
convergence statement for three-manifolds with initial $L^2$ norm of
curvature chosen small with respect only to the diameter and volume, which are
both necessary dependencies for a result of this kind.  In the critical
dimension
$n = 4$ we show a related low-energy convergence statement with an additional
hypothesis.  Finally we exhibit some finite time singularities in
dimension $n \geq 5$, and show examples of finite time singularities in
dimension $n \geq 6$ which are collapsed on the scale of curvature.
\end{abstract}

\date{\today}

\maketitle

\section{Introduction}

Let $M^n$ be a smooth compact manifold. Consider the
functional of Riemannian metrics
\begin{align} \label{Fdef}
\mathcal F(g) = \int_M \brs{\Rm_g}^2_g dV_g.
\end{align}
This is a natural analogue of the Yang-Mills energy for a Riemannian metric, and
studying its negative gradient flow,
\begin{gather} \label{flow}
\begin{split}
\dt g =&\ - \grad \FF\\
g(0) =&\ g_0,
\end{split}
\end{gather}
is a natural approach to understanding the structure of this functional. For
convenience below, we will call this the \emph{$L^2$ flow}. This is
a fourth-order, degenerate parabolic equation. Papers on this flow and closely
related topics include \cite{Bour}, \cite{SL21}, \cite{Yu}.  Certain
obstructions to the long
time existence of this flow have by now been established. For instance,
curvature blowup at the first singular time was established in \cite{SL2LTB}. As
in the case of Ricci flow, one key difficulty is to understand the possible
collapsing behavior at a singular time. Interestingly, it is a simple matter to
show that finite time singularities of the flow in dimensions $n = 2, 3$
\emph{must} be collapsed (see Proposition \ref{lowdimcollapse} below). The
contrapositive statement of Proposition \ref{lowdimcollapse} is that if one were
able to show a
noncollapsing result analogous to Perelman's estimate for Ricci flow \cite{P1},
one immediately concludes the long time existence of solutions to the $L^2$ flow
in dimensions $n = 2, 3$. Note that such a statement is plausible on PDE
grounds due to the ``supercriticality" of the functional $\mathcal F$ in those
dimensions, though perhaps counterintuitive due to the highly singular nature of
Ricci flow on three-manifolds.

Given the discussion above, let us remark on two cases where the noncollapsing
issue for solutions to the $L^2$ flow is well understood. The first is the case
of
Riemann surfaces.  A
compactness/concentration criterion originally due to Chen \cite{Chen}
states roughly that sequences of conformal metrics on Riemann surfaces either
converge or experience concentration volume and $L^1$ concentration of curvature
at a point. By
fixing a
special gauge to reduce the $L^2$ flow to a conformal flow, and by exploiting
some
energy estimates, we were able to rule
out the concentration behavior at finite time to establish long time existence
of the $L^2$ flow
on compact Riemann surfaces (\cite{SL2Surfaces} Theorem 1).

In higher dimensions the situation is more difficult since the flow cannot be
reduced to a conformal flow, and no convenient compactness criteria are
available to deal with the collapsing issue. One situation where this difficulty
was overcome is related to a certain conformal sphere theorem in four
dimensions. In \cite{CGY}, the authors show that a compact Riemannian
four-manifold with positive Yamabe constant and sufficiently small $L^2$ norm of
Weyl curvature tensor is smoothly deformable to a spherical space form. In
particular, the result yields a classification of the possible diffeotypes
satisfying the hypotheses. Moreover, the pinching condition on the Weyl tensor
is sharp. Recently, we showed a weaker version of this theorem using
the $L^2$ flow. Specifically, we showed that given a compact Riemannian
four-manifold with positive Yamabe constant and sufficiently small $L^2$ norm of
the traceless curvature tensor, the solution to the $L^2$ flow exists for all
time and converges to a spherical space form (\cite{SL23} Theorem 1). Later
this flow result was improved in \cite{Bour}, using a similar method, to yield
an explicit value for the required pinching, though this value is still
suboptimal.

A critical feature of the flow proofs mentioned in the above paragraph is that
the hypotheses of small $L^2$ norm of Weyl tensor and positive Yamabe constant
together imply a bound on the $L^2$ Sobolev constant. This is a beautiful
argument which combines the Gauss-Bonnet theorem and the solution to the Yamabe
problem, and is due independently to Gursky \cite{Gursky} and Tian.  In the 
method of \cite{SL23} and \cite{Bour} the Sobolev
constant bound is used to produce nonflat blowup limits of
finite time singularities, and these limits are automatically critical metrics
for the corresponding functionals.
Then, via a Liouville's Theorem argument one shows that noncompact critical
metrics satisfying certain Sobolev constant and $L^2$ curvature estimates are
automatically flat, implying the singularity could not have occurred, thus
yielding the long time existence of the flow. The argument
of Gursky-Tian perfectly resolves the crucial Sobolev constant issue in the
setting of positive Yamabe constant and pinched curvature. Outside of this
regime however this remains a difficult problem.

The purpose of this paper is to further flesh out and determine the nature of
collapse and
singularity formation in the $L^2$ flow, and moreover to highlight the crucial
role played by the dimension of the manifold in understanding this behavior. 
Our first result is a long-time
existence and convergence result for certain warped product $3$-manifolds.  The
theorem concerns the volume-normalized $L^2$ flow, defined in (\ref{vnflow}).

\begin{thm} \label{mainthm3flds} Let $(M^3, g)$ satisfy one of the following
conditions:
\begin{enumerate}
\item{$M \cong S^3$ and $g$ is an $SO(3)$-invariant metric.}
\item{$M \cong \Sigma \times S^1$, where $\Sigma$ is a compact Riemann surface,
$\chi(\Sigma) \neq 0$,
and $g = ds^2 + \psi^2 g_{\Sigma}$, where $g_{\Sigma}$ is a constant
curvature metric on $\Sigma$, and $\psi \in C^{\infty}(S^1, \mathbb R_{>0})$.}
\end{enumerate}
The solution to the volume-normalized $L^2$ flow with initial condition $g$
exists on $[0, \infty)$, and in the case of $SO(3)$-invariant metrics on $S^3$
converges to a critical metric.
\end{thm}

\begin{rmk} We note that the case $\chi(\Sigma) = 0$ is excluded because in the
course of the proof we need to ensure that if the area of one of the $\Sigma$
fibers approaches zero along the flow, its curvature goes to infinity.
\end{rmk}

\begin{rmk} 
The long time existence of $SO(3)$-invariant metrics on $S^3$ is surprising when
taken in contrast to the behavior of Ricci flow solutions in this context.
Indeed, as shown in \cite{Knopf}, neckpinches occur for metrics with such
symmetry. Moreover, as established in the work of Perelman \cite{P1}, \cite{P2},
neckpinches (degenerate and nondegenerate) are the only finite
time local singularities for Ricci flow on three-manifolds. From a PDE
perspective
though, this long time existence is not so surprising, since the functional
$\mathcal F$ is ``supercritical" in the scaling sense for $n \leq 3$, and hence
one expects very good
existence properties in these dimensions. Indeed, from this perspective it is
not so unreasonable to think that solutions to the $L^2$ flow on three-manifolds
always exist for all time (see Conjecture \ref{mainconj}). The fact that Theorem
\ref{mainthm3flds} applies to metrics
with arbitrary initial energy is encouraging in this respect.
\end{rmk}

\begin{rmk} It is reasonable to conjecture that the long time existence
statements of Theorem
\ref{mainthm3flds} hold for the gradient flows of more general
functionals, specifically quadratic Riemannian functionals satisfying
as $\mathcal F \geq \gd \int_M \brs{\Rm}^2$.
\end{rmk}

\noindent The next main result is a low energy convergence statement on
three-manifolds.

\begin{thm} \label{threelow} Given $V > 0, D > 0$, there exists $\ge > 0$ so
that
if $(M^3, g)$ is a compact Riemannian manifold satisfying
\begin{gather} \label{threelow10}
\begin{split}
\Vol(g) \geq&\ V,\\
\diam(g) \leq&\ D,\\
\mathcal F(g) \leq&\ \ge,
\end{split}
\end{gather}
the solution to the volume-normalized $L^2$ flow exists for all time and
converges to a flat metric.
\end{thm}

\begin{rmk} Neither the hypothesis of a lower bound on volume nor the hypothesis
of the
initial bound on diameter can be removed while keeping convergence to a flat
metric.  In particular, on $S^2 \times S^1$, the
metric $A g_{S^2} \oplus
\frac{1}{A^2} g_{S^1}$ has unit volume and $\diam = O( A^{\frac{1}{2}})$ and
$\mathcal F(g) = O( A^{-2})$ for $A$ large.  Since the universal cover of the
manifold is $S^2 \times \mathbb R$, it cannot admit a flat metric.  A
direct calculation shows that the solution to the $L^2$ flow with this initial
condition satisfies $A \to \infty$ as $ \to \infty$.  Likewise, the metrics
$g_{\ge} = g_{S^2} \oplus \ge^2 g_{S^1}$ have bounded diameter and $\mathcal
F(g_{\ge}) = O(\ge)$, $\Vol(g_{\ge}) = O(\ge)$.
\end{rmk}

The key estimate intervening in Theorem \ref{threelow}
controls the growth
of the first Laplace eigenvalue in the presence of a Sobolev constant bound
and small global energy. This estimate
bears a structural similarity to Perelman's $\gk$-noncollapsing estimate for
Ricci flow. Very roughly speaking, Perelman observes that if a solution to the
Ricci flow becomes sufficiently collapsed, there are test functions forcing his
quantity $\mu$ to approach $-\infty$. Since $\mu$ is monotonically increasing
along the flow, one thereby derives a contradiction. Our approach is similar in
that we control the Dirichlet energy $\EE(g, \phi)$ of a test function $\phi$
along the flow.  If
at 
a certain time $T$ the first Laplace eigenvalue is very small, one has a test
function $\phi$
for $\EE$
 which yields a very small, positive value. Taking a cue from
Perelman's conjugate heat equation, we push this function back to the initial
metric by means of the bi-Laplacian heat flow of the time varying metric and
derive a test function $\phi$ such that $\EE(g_0, \phi)$ is again a very small,
positive value. Considering the initial first Laplace eigenvalue as known, we
thus
derive an estimate for how fast it can decay along the flow. We emphasize that
the proofs really are only similar in their general outline. The energy $\EE$
is 
\emph{not} monotonic along solutions to the $L^2$ flow, and the main
difficulty is in controlling $\EE$ along the flow. Moreover, we
emphasize that we do \emph{not} show a $\gk$-noncollapsing estimate for
solutions to the $L^2$ flow akin to Perelman's estimate for Ricci flow. As
already mentioned, such an estimate immediately implies long time existence of
the flow on surfaces and three-manifolds (see Proposition \ref{lowdimcollapse}).

By adding an extra hypothesis, we obtain a low energy convergence statement on
four-manifolds as well.

\begin{thm} \label{fourlow1} Given constants $A, B > 0$ there exists $\ge(A, B)
> 0$ so that if
$(M^4, g)$ is a compact Riemannian manifold with unit volume satisfying
\begin{itemize}
\item{ $\nm{\grad \FF}{L^2} \leq A$}
\item{ $C_S(g) \leq B$, where $C_S(g)$ denotes the $L^2$ Sobolev constant of
$g$.}
\item{ $\int_M \brs{\Rm}^2 \leq \ge$}
\end{itemize}
then the solution to the $L^2$ flow with initial condition $g$ exists for all
time
and converges to a flat metric.
\end{thm}

\begin{rmk} It follows from Theorem \ref{fourlow1} that the difficulty in
proving a version of Theorem \ref{threelow} for four-manifolds lies entirely in
understanding the short-time behavior of the Sobolev constant.  In particular,
if one could obtain a
doubling-time estimate for the Sobolev
constant along the $L^2$ flow, it would follow that for some time along the flow
the metric satisfies a bound on the Sobolev constant and the $L^2$ norm of
$\grad \FF$.  The analogue of Theorem
\ref{threelow} for four-manifolds would then follow, i.e. given energy
sufficiently small with respect to the Sobolev constant, the solution to the
$L^2$ flow with this initial condition would exist for all time and converge to
a flat metric.
\end{rmk}

Moving to higher dimensions, we show in \S \ref{highdim} that in all
dimensions $n \geq 5$, the $L^2$ flow exhibits finite time singularities, and in
dimensions $n \geq 6$, the $L^2$ flow exhibits finite time singularities which
do
not satisfy an injectivity radius estimate on the scale of curvature.  Also note
that we know from Proposition \ref{lowdimcollapse} that finite time
singularities in dimension $n = 2,3$ \emph{must} be collapsed on the scale of
maximum curvature.  On the other hand, as discussed above, it is not
unreasonable to expect that there are \emph{no} finite time singularities of the
$L^2$ flow in dimensions $n = 2,3$.
Thus one way to interpret Proposition \ref{lowdimcollapse} is that if one could
exhibit a no-local-collapsing result for the $L^2$ flow akin to Perelman's
result for Ricci flow, the general long time existence in dimensions $n = 2,3$
would immediately follow.  What the higher dimensional examples show is that
such a general noncollapsing result can only hold in dimension $n \leq 5$.

Here is an outline of the rest of the paper. In $\S$ \ref{sobbckgrnd} we give
some background on Sobolev and isoperimetric constants, and in $\S$
\ref{flowbckgrnd}
we give background results on the $L^2$ flow.  Section \ref{threesym} has the
proof of Theorem \ref{mainthm3flds}.  In \S \ref{dirichletgrowth} we derive an
estimate for the growth of the first Laplace eigenvalue along solutions to the
$L^2$ flow, and we use this in \S \ref{threelowsec} to prove Theorem
\ref{threelow}.  Section
\ref{fourmnflds} contains the proof of Theorem \ref{fourlow1}.  In \S
\ref{highdim} we address the
behavior of solutions to the $L^2$ flow in dimensions $n \geq 5$.  We end in
\S \ref{conjsec} with a conjectural discussion of the optimal long time
existence results for the $L^2$ flow and their possible applications.

\textbf{Acknowledgments:} The author would like to thank Aaron Naber, Peter
Petersen, and Laurent Saloff-Coste for several helpful discussions.

\section{Isoperimetric and Sobolev Constants} \label{sobbckgrnd}

In this section we recall some definitions and theorems related to isoperimetric
and Sobolev constants on compact Riemannian manifolds.

\begin{defn} Let $(M^n, g)$ be a compact Riemannian manifold, and let $\Omega$
denote a proper open subset of $M$  The \emph{isoperimetric constant} is
\begin{align*}
 C_I(M, g) := \inf_{\Omega} \frac{\Area(\del \Omega)}{\min \{ \Vol(\Omega),
\Vol(M \backslash \Omega) \}^{\frac{n-1}{n}}}.
\end{align*}
\end{defn}

\begin{defn} Let $(M^n, g)$ be a compact Riemannian manifold.  The \emph{$L^1$
Sobolev constant} of $M$ is the infimum of all $C$
such that for any $f \in C^1(M)$,
\begin{align*}
\inf_{\ga \in \mathbb R} \left( \int_M \brs{f - \ga}^{\frac{n}{n-1}} dV
\right)^{\frac{n-1}{n}} \leq C \int_M \brs{\N f} dV.
\end{align*}
\end{defn}

\begin{defn} \label{Sobconstdef} Let $(M^n, g)$ be a compact Riemannian
manifold, $n \geq 3$.
The
\emph{$L^2$ Sobolev constant}, denoted $C_S(g)$, is the infimum
of all $C$ such that for any $f \in C^1(M)$,
\begin{align*}
\left( \int_M f^{\frac{2n}{n-2}} dV \right)^{\frac{n-2}{n}} \leq C \left(\int_M
\brs{\N f}^2 dV + V^{-\frac{2}{n}} \int_M f^2 dV \right).
\end{align*}
\end{defn}

\begin{rmk} \label{isoequiv} The $L^1$ Sobolev inequality is equivalent to the
isoperimetric
constant by \cite{Bombieri} (see also \cite{Chavel}).  Furthermore, by
(\cite{Li} Lemma 2) an upper bound on the $L^1$ Sobolev constant implies an
upper
bound for the $L^2$ Sobolev constant.
\end{rmk}

\begin{thm} \label{multsob} Let $(M^n, g)$ be a complete Riemannian manifold, $n
\geq 3$. Given $q \geq 2$, there exists $C(q)$ such that for all $u \in
H_1^q(M)$,
\begin{align*}
\nm{u}{L^p} \leq C C_S^{\tfrac{\ga}{2}} \nm{u}{L^m}^{1-\ga} \left( \nm{\N
u}{L^q} +
\nm{u}{L^q} \right)^{\ga},
\end{align*}
where $2 \leq m \leq p$, 
\begin{align*}
\ga = \frac{ \frac{1}{m} - \frac{1}{p}}{\frac{1}{m} - \frac{1}{q} +
\frac{1}{n}},
\end{align*}
and
\begin{align*}
\begin{cases}
\mbox{ if $q < n$, then $p \leq \frac{nq}{n-q}$ and $C = C(n, q)$},\\
\mbox{ if $q = n$, then $p < \infty$ and $C = C(m, p)$},\\
\mbox{ if $q > n$, then $p \leq \infty$ and $C = C(n, m, q)$.}
\end{cases}
\end{align*}
\begin{proof} See \cite{Lady} or \cite{Bour} for a more recent exposition.
\end{proof}
\end{thm}

\begin{lemma} \label{diambnd} Given $A > 0$ there exists $D = D(A) > 0$ so that
if $(M^n, g)$ is a compact Riemannian manifold with $\Vol(g) = 1$ and $C_S(g) <
A$,
then $\diam(g) < D$.
\begin{proof} It follows from the argument of (\cite{Hebey} Lemma 3.2) that for
$(M^n, g)$ satisfying the hypotheses, for any $x \in M$ one has
\begin{align} \label{volbnd}
\Vol(B_x(1)) \geq \frac{C}{A^{\frac{n}{2}}}
\end{align}
for a universal constant $C$. The result follows from a standard packing
argument.
\end{proof}
\end{lemma}

\begin{rmk} A much more robust argument for this lemma, showing the relationship
of Sobolev inequalities and diameter, appears in \cite{BL}.
\end{rmk}

\begin{defn} \label{gldef} Let $(M^n, g)$ be a compact Riemannian manifold.
Consider the
functional
\begin{align*}
\mathcal E(\phi, g) := \frac{\nm{\N \phi}{L^2}^2}{\nm{\phi}{L^2}^2}.
\end{align*}
The \emph{first Laplace eigenvalue} of $g$ is
\begin{align*}
\gl(g) =&\ \inf_{ \{\phi | \int_M \phi dV = 0 \}} \mathcal E(\phi, g).
\end{align*}
\end{defn}

\section{Background on the \texorpdfstring{$L^2$}{L2} flow} \label{flowbckgrnd}

In this section we collect some facts about the functional $\mathcal F$ and
solutions to the $L^2$ flow. First we recall (\cite{Besse} Chapter 4.H) that
\begin{align*}
\grad \FF =&\ 2 \gd d \Rc - 2 \check{R} + \frac{1}{2} \brs{\Rm}^2 g
\end{align*}
where $d$ is the exterior derivative induced by the Levi-Civita connection on
$\Lambda^1 \otimes \Lambda^1$, $\gd$ is the $L^2$-adjoint of $d$, and
\begin{align*}
\check{R}_{ij} =&\ R_i^{pqr} R_{jpqr}.
\end{align*}
From this it follows that, if $s$ denotes the scalar curvature, and $\gD = \tr_g
\N^2$,
\begin{align} \label{gradF}
\tr \grad \FF =&\ - \gD s + 2 \left( \frac{n}{4} - 1 \right) \brs{\Rm}^2.
\end{align}

\noindent Next we note certain elliptic and parabolic coercivity estimates which
we will use in the estimates
below.

\begin{lemma} \label{coercive} (\cite{SL22} Lemma 2.2) There exists a universal
constant $C$ such that if $(M^4, g)$ is a compact four-manifold, then
\begin{align*}
\nm{\N^2 \Rm}{L^2}^2 \leq C \left[ \nm{\grad \FF}{L^2}^2 + \int_M \brs{\N
\Rm}^2\brs{\Rm} \right].
\end{align*}
\end{lemma}

\begin{prop} \label{Egrowthprop} (\cite{SL22} Proposition 4.6) Given $A > 0$
there exists $\ge > 0$ so that if
$(M^4, g_t)$ a solution to the $L^2$ flow on $[0, T]$ such that
\begin{align*}
\sup_{t \in [0, T]} C_S(g_t) \leq&\ A,\\
\FF(g_0) \leq&\ \ge,
\end{align*}
then
\begin{align*}
\sup_{[0, T]} \nm{\grad \FF}{L^2}^2 + \int_0^T \nm{\N^2 \grad \FF}{L^2}^2 \leq&\
2 \nm{\grad \FF}{L^2(g_0)}^2 + C A^2 \ge^{\frac{1}{4}}.
\end{align*}
\end{prop}

Next we recall a fundamental existence result on the existence and long time
behavior of solutions to the $L^2$ flow.
\begin{thm} \label{existencecor}  (\cite{SL2LTB} Corollary 1.9)
Let $(M^n, g)$ be a compact Riemannian manifold.  The
solution to the $L^2$ flow with initial condition $g$ exists on a maximal time
interval $[0, T)$.  Furthermore, if $T < \infty$ then
\begin{align*}
\limsup_{t \to T} \brs{\Rm}_{g(t)} = \infty.
\end{align*}
\end{thm}

\begin{thm} (\cite{SL2LTB} Theorem 1.3) \label{globalsmoothing} Fix $m, n \geq
0$.  There exists a constant
$C = C(m, n)$ so that if $(M^n, g(t))$ is a complete solution to the $L^2$ flow
on
$\left[0, T \right]$ satisfying
\begin{align*}
\sup_{M \times \left[0, T \right]} \brs{\Rm} \leq K,
\end{align*}
then
\begin{align} \label{globalsmoothingestimate}
\sup_{M} \brs{\N^m \Rm} \leq C \left( K +
\frac{1}{t^{\frac{1}{2}}} \right)^{1 + \frac{m}{2}}.
\end{align}
\end{thm}

\begin{cor} \label{wkcompactness} (\cite{SL2LTB} Corollary 1.5)
Let $\{(M^n_i, g_i(t), p_i)\}$ be a sequence of complete pointed
solutions of the $L^2$ flow, where $t \in (\ga, \gw), - \infty
\leq \ga < \omega \leq \infty$.  Suppose there exists $K < \infty$ and $\gd > 0$
such that
\begin{align*}
\sup_{M_i \times (\ga, \gw)} \brs{\Rm(g_i)}_{g_i} \leq K, \qquad
\inj_{g_i(0)}(p_i) \geq \gd.
\end{align*}
Then there exists a subsequence $\{(M_{i_j}, g_{i_j}(t), p_{i_j}) \}$ and a one
parameter family of Riemannian manifolds $(M_{\infty}, g_{\infty}(t),
p_{\infty})$ such
that $\{(M_{i_j}, g_{i_j}, p_{i_j}) \}$ converges
to
$(M_{\infty}, g_{\infty}(t), p_{\infty})$ in the $C^{\infty}$ Cheeger-Gromov
topology.
\end{cor}

Next we make the observation mentioned in the introduction, namely that for
solutions in dimensions $n = 2, 3$ any finite time singularity must be collapsed
on the scale of curvature.

\begin{prop} \label{lowdimcollapse} Let $(M^n, g(t))$ be a solution to
the $L^2$ flow, $n = 2, 3$. Suppose $g(t)$ exists on a maximal time interval
$[0,
T)$, $T < \infty$. Let $\{(x_i, t_i) \}$ be a sequence of points such that $t_i
\to T$ and
\begin{align*}
\brs{\Rm}(x_i, t_i) = \sup_{[0, t_i]} \brs{\Rm}.
\end{align*}
Then
\begin{align*}
\lim_{i \to \infty} \inj_g(x_i) \brs{\Rm}(x_i) = 0.
\end{align*}
\begin{proof} Suppose the claim is false, and let $\{(x_i, t_i) \}$ be the
sequence of points as in the statement. Let $\gl_i = \brs{\Rm}(x_i, t_i)$, and
let
\begin{align*}
\til{g}^i = \gl_i g \left(t_i + \frac{t}{\gl_i^2} \right).
\end{align*}
A direct calculation shows that $\til{g}^i$ is a solution to the $L^2$ flow on
$[-t_i \gl_i^2, 0]$ with bounded curvature. Moreover, since we have assumed
$\lim_{i \to \infty} \inj(x_i) \brs{\Rm}(x_i) > 0$, we conclude $\lim_{i \to
\infty} \inj_{\til{g}^i}(x_i) > 0$. It follows from Corollary
\ref{wkcompactness} that the sequence of solutions $\{M, g^i, x_i \}$ contains a
subsequence
converging to a smooth, nonflat solution to the $L^2$ flow which we denote
$(M_{\infty}, g^{\infty}, x_{\infty})$. But since $n < 4$, we note that
$\mathcal F(\til{g}^i(0)) = \FF(\gl_i g(t_i)) = \gl_i^{\frac{n}{2} - 2}
\FF(g(0)) \to 0$. Thus $g^{\infty}$ must be flat, a contradiction.
\end{proof}
\end{prop}

\begin{rmk} Note that we need the improved compactness theorem proved recently
in (\cite{SL2LTB} Corollary 1.5) to obtain this result, prior results requiring
a global
Sobolev constant bound will not suffice here.
\end{rmk}

We conclude with some remarks on the volume normalized version of the $L^2$
flow. 
It
follows from (\ref{gradF}) that if $(M^n, g(t))$ is a solution to the $L^2$ flow
then
\begin{align*}
\dt \Vol(g(t)) =&\ \frac{4 - n}{4} \mathcal F(g(t)).
\end{align*}
In particular, the initial value problem
\begin{gather} \label{vnflow}
\begin{split}
\dt g =&\ - \grad \FF + \frac{n-4}{2n} \frac{\mathcal F(g)}{\Vol(g)} \cdot
g\\
g(0) =&\ g_0,
\end{split}
\end{gather}
preserves the volume of the time dependent metrics, and we call it the
\emph{volume-normalized $L^2$ flow}.  One can check that for an
initial metric $g_0$, the corresponding solutions to the $L^2$ flow and
the volume normalized $L^2$ flow differ by a rescaling in space and time. 
Furthermore, the volume normalized $L^2$ flow is the gradient flow of the
functional
\begin{gather} \label{Ftildef}
 \til{\FF}(g) = \Vol(g)^{\frac{4-n}{n}} \FF(g)
\end{gather}

\begin{defn} \label{nonsingdef} A solution $(M^n, g(t))$ to the volume
normalized $L^2$ flow is \emph{nonsingular}
if
it exists on $[0, \infty)$ with a
uniform bound on the curvature tensor.
\end{defn}

\begin{thm} (\cite{SL2LTB} Theorem 1.16) \label{nonsingularthm} Let $(M^n,
g(t))$ be a nonsingular solution
to the volume normalized $L^2$ flow.  Then either
\begin{itemize}
\item{ For all $p \in M$, $\limsup_{t \to \infty} \inj_p g(t) = 0$.}
\item{ There exists a sequence of times $t_i \to \infty$ such that $\{g(t_i) \}$
converges to a smooth metric on $M$ which is critical for $\til{\FF}$.}
\item{ There exists a sequence of points $(p_i, t_i), t_i \to \infty$ such that
$\{(M, g(t_i), p_i) \}$
converges to a complete noncompact finite volume metric which is critical for
$\til{\FF}$.}
\end{itemize}
\end{thm}

\section{Three-manifolds with symmetry} \label{threesym}

In this section we study solutions to the $L^2$ flow where the initial condition
is either a warped product of a surface with constant curvature over $S^1$, or
an
$SO(3)$-invariant metric on $S^3$. These two cases are natural to combine since
$SO(3)$-invariant metrics on $S^3$ are equivalent to warped product metrics on
$(-1, 1) \times S^2$ with certain boundary conditions prescribed below.
Specifically,
fix $\Sigma$ a compact
Riemann surface, $\chi(\Sigma) \neq 0$, and let $g_{\Sigma}$ denote a metric of
constant curvature $-1, 1$ depending on the Euler characteristic of $\Sigma$.
Let $M \cong (-1, 1) \times \Sigma$.  Let $\phi, \psi : (-1, 1) \to \mathbb R_{>
0}$, and consider
the Riemannian metric
\begin{gather} \label{warped}
g = \phi(x)^2 dx^2 + \psi(x)^2 g_{\Sigma}.
\end{gather}
It will frequently be useful to use the natural geometric coordinate of
lateral distance from the slice $\Sigma \times \{0\}$.  In particular, set
\begin{gather} \label{sdef}
s(x) = \int_0^x \phi(w) dw.
\end{gather}
The range of $s$ is some interval we will always refer to as $(a, b)$. Moreover,
let
\begin{align*}
L(g) := \int_{-1}^1 \phi(x) dx.
\end{align*}
$L$ represents the length of the base circle in the case of a product topology,
or the distance from the north pole to the south pole in the case of $SO(3)$
invariant metrics on $S^3$.  Of course $L = b - a$.  Using
the parameter $s$, the metric (\ref{warped}) takes the form
\begin{gather} \label{warpeds}
g = ds^2 + \psi(s)^2 g_{\Sigma}.
\end{gather}
If we impose periodic boundary conditions, i.e. that $\phi$ and $\psi$ extend
to smooth functions on $S^1$, then $g$ defines a metric on $S^1 \times \Sigma$.
 To produce $SO(3)$-invariant metrics on $S^3$, we will impose
\begin{gather} \label{spherebndry}
\begin{split}
\lim_{x \to \pm 1} \psi =&\ 0\\
\lim_{x \to \pm 1} \psi_s =&\ \mp 1.
\end{split}
\end{gather}

Next we establish some geometric estimates for the Riemannian
manifolds described above. We will informally refer to
these as warped products, with the assumption that the
boundary
conditions have been chosen appropriate to the given topology.
\begin{lemma} Let $(M^3, g)$ be a warped product. Let $v$ and
$w$
denote vectors in $\pi^* T\Sigma$.  Then
\begin{align} \label{warpedcurvature}
K_1 = K\left(\frac{\del}{\del s} \wedge v \right) =&\ - \frac{\psi_{ss}}{\psi},
\qquad K_2 = K(v \wedge w) = \frac{K_{\Sigma} - \psi_s^2}{\psi^2}.
\end{align}
\end{lemma}

\begin{lemma} Let $(M^3, g)$ be a warped product.  Then
\begin{align} \label{volumecalc}
\Vol(g) =&\ \Vol(g_{\Sigma}) \int_a^b \psi^2 dw,\\ \label{L2calc}
\int_M \brs{\Rm}^2 dV =&\ \Vol(g_{\Sigma}) \int_a^b \left( 4 \psi_{ss}^2 + 2
\frac{\left(K_{\Sigma} - \psi_s^2 \right)^2}{\psi^2} \right) dw.
\end{align}
\begin{proof} We directly compute
\begin{align*}
\Vol(g) =&\ \int_M dV_g = \int_a^b \int_{\Sigma} \psi^2 d\Sigma dw =
\Vol(g_{\Sigma}) \int_a^b \psi^2 dw.
\end{align*}
Next, using (\ref{warpedcurvature}) we compute
\begin{align*}
\int_M \brs{\Rm}^2 =&\ \int_a^b \int_{\Sigma} \left( 4
\frac{\psi_{ss}^2}{\psi^2} + 2 \frac{\left(K_{\Sigma} - \psi_s^2
\right)^2}{\psi^4} \right) \psi^2 d\Sigma dw\\
=&\ \Vol(g_{\Sigma}) \int_a^b \left(4 \psi_{ss}^2 + 2 \frac{\left(K_{\Sigma} -
\psi_s^2 \right)^2}{\psi^2} \right) dw.
\end{align*}
\end{proof}
\end{lemma}

\begin{lemma} \label{noncollapsesphere} There is a constant $\gd > 0$ so that if
$(S^3, g)$ is an $SO(3)$-invariant metric with $\brs{\Rm} \leq 1$, then $L \geq
\gd$.
\begin{proof} Recall that our metric satisfies (\ref{spherebndry}).  Let $\mu =
\sup \{t > 0 | \frac{1}{2} \leq \psi_s(a + t) \leq 2 \}$. Clearly $L \geq \mu$,
thus it
suffices to bound $\mu$ from below.  First note that on $[a, a + \mu]$ certainly
$\psi \leq 2\mu$.  Without loss of generality assume $\mu \leq
\frac{1}{8}$, so that $\psi \leq \frac{1}{4}$ on $[0, \mu]$. Now observe that
since $\brs{\Rm} \leq 1$,
\begin{align*}
\psi_{ss} = - \psi K_1 \geq - \psi \geq -\frac{1}{4}.
\end{align*}
It follows that, on $[0, \mu]$, $\psi_s \geq 1 - \frac{\mu}{4}$. Likewise we can
estimate on $[0, \mu]$
\begin{align*}
\psi_s^2 = K_{\Sigma} - K_2 \psi^2 \leq 1 + \psi^2 \leq 1 + \left( 2 \mu
\right)^2 \leq 1 + \frac{1}{4}.
\end{align*}
This implies a lower bound for $\mu$, and the lemma follows.
\end{proof}
\end{lemma}

\begin{prop} \label{noncollapseprop} Let $(M^3, g)$ be a warped product with
fiber $\Sigma$, $\chi(\Sigma) \neq 0$, further satisfying
$\brs{\Rm} \leq
1$.  Given $\ge > 0$, there exists $\gd > 0$ so that if $L_g \geq \ge$ then
\begin{align} \label{noncollapse}
\inj_g \geq \gd.
\end{align}
\begin{proof} We first consider the case of product topology, i.e. $M \cong
S^1 \times \Sigma$.  Fix $s_0 \in (a, b)$ a minimum point for $\psi$.  At this
point we compute
\begin{align*}
1 \geq \brs{K_2} =&\ \brs{\frac{K_{\Sigma} - \psi_s^2}{\psi^2}} =
\frac{\brs{K_{\Sigma}}}{\psi^2}.
\end{align*}
We conclude that for all $s$,
\begin{align*}
\psi(s) \geq \psi(s_0) \geq \brs{K_{\Sigma}} = 1
\end{align*}
since $\chi(\Sigma) \neq 0$.  We now show a lower volume growth estimate for
sufficiently small balls. Fix a constant $r_0 > 0$ so that $r_0 \leq
\inj(g_{\Sigma})
$, and also $r_0 \leq \frac{\ge}{2} \leq \frac{L}{2}$. Now fix an arbitrary
$(p_0, s_0) \in \Sigma
\times S^1$ and consider $B_{r}(p_0, s_0)$.  We want to show that there is a
uniform constant $\mu > 0$ so that, for all $r \leq r_0$,
\begin{align*}
\frac{\Vol(B_r(p_0, s_0))}{r^3} \geq \mu.
\end{align*}
Without loss of generality we can reparameterize $s$ so that $s_0 = 0$, and the
range of $s$ is $\left( - \frac{L}{2}, \frac{L}{2} \right)$. First we claim the
inclusion
\begin{align} \label{Uinc}
U := B_{\frac{r}{2}, \psi^2(s_0) g_{\Sigma}}(p_0) \times \left[s_0 -
\frac{r}{2},
s_0 + \frac{r}{2} \right] \subset B_{r, g} (p_0, s_0).
\end{align}
To show this let $(q, t) \in U$ and let $\gg$ denote the curve which is the
concatenation of the shortest geodesic in $\Sigma$ connecting $p_0$ and $q$, in
the metric $\psi^2(s_0) g_{\Sigma}$, with the lateral curve connecting $(q,
s_0)$ to $(q, t)$.  One has
\begin{align*}
d_{g} ((p_0, s_0), (q, t)) \leq&\ \Length(\gg) \leq \frac{r}{2} + \frac{r}{2} =
r.
\end{align*}
Therefore the inclusion (\ref{Uinc}) holds.  One can then compute
\begin{align*}
\Vol(B_r(p_0,s_0)) \geq&\ \Vol(U)\\
=&\ \int_{s_0 - \frac{r}{2}}^{s_0 + \frac{r}{2}} \int_{B_{\frac{r}{2},
\psi^2(s_0) g_{\Sigma}}(p_0, s_0)} \psi^2(w) d\Sigma dw\\
\geq&\ \inf \psi^2 \int_{s_0 - \frac{r}{2}}^{s_0 + \frac{r}{2}}
\int_{B_{\frac{r}{2}, \psi^2(s_0) g_{\Sigma}}(p_0, s_0)} d\Sigma dw\\
\geq&\ \frac{\nu}{4} r^3.
\end{align*}
where $\nu$ is a lower bound on the volume growth of $g_{\Sigma}$.  The lower
bound on volume growth follows, and by Cheeger's lemma the proposition
follows.

For the case of $SO(3)$-invariant metrics on $S^3$, we first consider the north
and south poles.  By a direct argument as in Lemma \ref{noncollapsesphere}, we
can obtain positive lower bound for $\psi_s$ in some controlled neighborhood
around $s = a$.  This implies a linear lower bound on the growth of $\psi$ near
$s = a$ and then by a direct integration we can obtain the requisite volume
lower bound near $a$.  In fact this argument produces a volume lower bound for
points near $a$ as well, and a directly analogous bound takes care of points
near $s = b$.  For points in the interior the argument is the same as that above
for product topologies, finishing the proof.
\end{proof}
\end{prop}

\begin{proof}[Proof of Theorem \ref{mainthm3flds}] Let $(M^3, g(t))$ be a
solution to the volume normalized $L^2$ flow as in the
statement.  The first step is to show long time existence.  We know from Theorem
\ref{existencecor} (which applies to the volume normalized flow by a simple
rescaling argument) that if the maximal existence time is $T < \infty$ then 
\begin{align*}
\limsup_{t \to T} \brs{\Rm}_{g(t)} = \infty.
\end{align*}
Choose a sequence of points $(x_i, t_i)$ such that
\begin{align*}
\gl_i := \brs{\Rm(x_i)}_{g(t_i)} = \sup_{M \times [0, t_i]} \brs{\Rm}.
\end{align*}
Let
\begin{align*}
g_i(t) :=&\ \gl_i g \left(t_i + \frac{t}{\gl_i^2} \right)
\end{align*}
By construction one notes that the solution $g_i(t)$ exists on $[-\gl_i^2 t_i,
0]$, and moreover
\begin{align} \label{blowupcurvbnd}
\sup_{M \times [-\gl_i^2 t_i, 0]} \brs{\Rm} = \brs{\Rm(x_i)}_{g_i(0)} = 1.
\end{align}
We want to take a convergent subsequence of these solutions $\til{g}_i$, and to
obtain a manifold in the limit we require a lower bound on the injectivity
radius.  
By Proposition \ref{noncollapseprop} it suffices to show a lower bound on the
lateral distance $L$.

Here we break into cases. In the case of $SO(3)$-invariant metrics on $S^3$,
Lemma \ref{noncollapsesphere} provides the required lower bound. For the
remaining cases we argue by contradiction.  Assume $L_i \to 0$. 
Note that the solutions $(M, g_i(t))$ exist, for sufficiently large $i$, on
$[-1, 0]$. It follows from (\ref{blowupcurvbnd}) and Theorem
\ref{globalsmoothing} that there is a uniform
constant $C$ such that
\begin{align*}
\sup_{M} \brs{\N \Rm}_{g_i(0)} \leq C.
\end{align*}
Say the point $x_i$ is given by $(p_i, s_i) \in \Sigma \times (a, b)$. By
integrating along a lateral geodesic one concludes
\begin{align*}
\brs{\Rm}_{g_i(0)} (p_i, s_i \pm \tau) \geq \brs{\Rm}_{g_i(0)} (p_i, s_i) - \tau
C \geq 1 - C L_i.
\end{align*}
Since the
curvatures are functions of the parameter $s$ only, and since $L_i \to 0$ we may
conclude that for sufficiently large $i$ one has
\begin{align} \label{lowercurvbnd}
\inf_{M} \brs{\Rm}_{g_i(0)} \geq \frac{1}{2}.
\end{align}
Since the volume of the unscaled metrics was fixed, for the rescaled solutions
it follows that
\begin{align} \label{lowervolumeblowup}
\Vol(g_i(0)) \geq&\ \Vol(g(0)) \gl^{\frac{3}{2}}.
\end{align}
Combining (\ref{lowercurvbnd}) and (\ref{lowervolumeblowup}) we see that
\begin{align} \label{energyblowup}
\int_M \brs{\Rm}^2_{g_i(0)} dV_{g_i(0)} \geq \frac{\Vol(g(0))}{4}
\gl_i^{\frac{3}{2}} \to \infty
\end{align}
as $i \to \infty$.  But of course $\int_M \brs{\Rm}^2(g(t)) dV \leq C$ thus
\begin{align*}
\int_M \brs{\Rm}^2_{g_i(0)} dV_{g_i(0)} =&\ \gl_i^{-\frac{1}{2}} \int_M
\brs{\Rm}^2_{g(t_i)} dV_{g(t_i)} \to 0,
\end{align*}
contradicting (\ref{energyblowup}).  It follows that there is a uniform
constant$\mu > 0$ so that
\begin{align*}
L(g_i(0)) \geq \mu > 0.
\end{align*}
We now conclude from Proposition \ref{noncollapseprop} that there
is a constant $\gd > 0$ independent of $i$ so that
\begin{align*}
\inj_{g_i(0)}(x_i) > \gd.
\end{align*}
Using this and (\ref{blowupcurvbnd}), we conclude from Theorem
\ref{wkcompactness}
that there is a subsequence of $\{(M, g_i(t), x_i) \}$ converging to a pointed
solution $(M^3_{\infty}, g_{\infty}(t), x_{\infty})$ to the volume normalized
$L^2$ flow.  By construction this
solution satisfies
\begin{align} \label{limitcurvature}
\brs{\Rm}_{g_{\infty}(0)}(x_{\infty}) = 1.
\end{align}
However, one has
\begin{align*}
\lim_{i \to \infty} \int_M \brs{\Rm}^2_{g_i(0)} dV_{g_i(0)} =&\ \lim_{i \to
\infty} \gl_i^{-\frac{1}{2}} \int_M
\brs{\Rm}^2_{g(t_i)} dV_{g(t_i)}\\
\leq&\ \lim_{i \to \infty} \gl_i^{-\frac{1}{2}} \int_M \brs{\Rm}^2_{g(0)}
dV_{g(0)}\\
=&\ 0.
\end{align*}
It follows by Fatou's Lemma that
\begin{align*}
\int_{M_{\infty}} \brs{\Rm}^2_{g_{\infty}(0)} dV_{g_{\infty}(0)} = 0,
\end{align*}
contradicting (\ref{limitcurvature}).  It follows that the curvature is bounded
on finite time intervals, and therefore the solution exists on $[0, \infty)$. To
show the uniform curvature bound, one can repeat the argument by
contradiction above, again blowing up around a sequence of points realizing the
spacetime maximum of curvature.

Turning now to the convergence statement for $SO(3)$-invariant metrics, we note
that we have shown that the
solutions under consideration are nonsingular in the sense of Definition
\ref{nonsingdef}.  We
must consider the different possibilities for the limiting behavior given by
Theorem \ref{nonsingularthm}. First let us globally rescale the solution to
(\ref{vnflow}) so that
\begin{align*}
\sup_{M \times [0, \infty)} \brs{\Rm} \leq 1.
\end{align*}
Using Lemma \ref{noncollapsesphere} and Proposition \ref{noncollapseprop} we
conclude that
\begin{align*}
\inf_{M \times [0, \infty)} \inj_g \geq \gd > 0.
\end{align*}
It follows that the first and third possibilities of Theorem
\ref{nonsingularthm}
are impossible, therefore we must have subsequential convergence to a critical
metric.
\end{proof}

\section{Estimates of First Laplace Eigenvalue}\label{dirichletgrowth}

The main purpose of this section is to prove Theorem \ref{glestimate} below,
which is
an estimate on the decay of the first Laplace eigenvalue of the evolving
metric along a solution to the $L^2$ flow.  The strategy of the proof is to take
a
test function for the functional $\EE$ at
a certain forward time in the flow, then push it back to the initial time using
the backwards biharmonic heat flow. 

\begin{thm} \label{glestimate} There exist universal constants $C > 0$,
$\bar{\gl} > 0$
such that given $A
> 1$, there exists $\ge > 0,$ so that if $(M^n, g_t)$ is a solution to
the $L^2$ flow on a compact manifold $M^n$, $n \leq 4$, on a time interval $[0,
T]$, $T \leq 1$, satisfying
\begin{enumerate}
\item{$\mathcal F(g_0) \leq \ge$}
\item{For all $f \in C^1(M)$, all $t \in [0, T]$, $\nm{f}{L^4}^2 \leq A \left(
\nm{\N f}{L^2}^2 + \nm{f}{L^2}^2 \right)$}
\item{$\gl(g_T) \leq \bar{\gl}$.}
\end{enumerate}
Then
\begin{align*}
 \gl(g_0) \leq&\ 2 \gl(g_T) + C A^2 \ge^{\frac{1}{2}}.
\end{align*}
\end{thm}

\begin{rmk} The second hypothesis above is of course just the usual $L^2$
Sobolev inequality when $n = 4$, but we have restated it here to unify the
discussions for all dimensions $n \leq 4$, which will simplify the proof below.
\end{rmk}

Without further comment we fix throughout this section a solution $(M^n, g_t)$
satisfying the hypotheses above, the notation of Theorem \ref{glestimate}, and
the notation in the following paragraph.
As a
notational convenience we set
\begin{align*}
E := \dt g.
\end{align*}
Furthermore, let $\tau = T - t$ and define $\phi$ on $M \times [0, T]$ via
\begin{gather} \label{phiflow}
\begin{split}
\frac{\del}{\del \tau} \phi =&\ - \gD^2_{g_{\tau}} \phi + \frac{1}{2} \phi \tr_g
E\\
\phi(T) =&\ \phi_T.
\end{split}
\end{gather}
To clarify, as written the first equation above is parabolic in the backwards
time parameter $\tau$,
therefore we specify a \emph{final} value and solve backwards in time. As this
is a linear parabolic equation defined with respect to a smooth one-parameter
family of metrics, the existence of $\phi$ on the whole interval $[0, T]$
follows from standard estimates for linear parabolic equations. In deriving the
estimates below we adopt the convention that $C$ always denotes a universal
constant, which may change from line to line. First we observe that the
$L^1$ norm of $\phi$ is fixed for all times, which is the purpose of inserting
the zeroth order term into (\ref{phiflow}).  

\begin{lemma} \label{phiflowfixing} One has
\begin{align*}
\frac{\del}{\del \tau} \int_M \phi dV =&\ 0
\end{align*}
\begin{proof} We directly compute
\begin{align*}
\frac{\del}{\del \tau} \int_M \phi dV = \int_M \left( \frac{\del}{\del \tau}
\phi - \phi \frac{1}{2} \tr_g E \right) dV =  \int_M - \gD^2 \phi dV = 0.
\end{align*}
\end{proof}
\end{lemma}

\noindent In the next two lemmas we
derive differential inequalities for the
evolutions of $\nm{\phi}{L^2}$ and $\nm{\N \phi}{L^2}^2$.

\begin{lemma} \label{Sgrowthlemma5} One has
\begin{align*}
\frac{\del}{\del \tau} \nm{\phi}{L^2}^2 =&\ - 2 \nm{\gD \phi}{L^2}^2 + \int_M
\left[ \frac{1}{2} \tr_g E \phi^2 \right] dV
\end{align*}
\begin{proof} We compute
\begin{align*}
\frac{\del}{\del \tau} \nm{\phi}{L^2}^2 =&\ 2 \int_M \phi \left( - \gD^2 \phi +
\phi \frac{1}{2} \tr_g E\right) + \int_M
\left[ - \frac{1}{2} \tr_g E \phi^2 \right] dV\\
=&\ - 2 \nm{\gD \phi}{L^2}^2 + \int_M \left[ \frac{1}{2} \tr_g E \phi^2 \right]
dV.
\end{align*}
\end{proof}
\end{lemma}

\begin{lemma} \label{Sgrowthlemma10} There exists a universal constant $C$ such
that
\begin{align} \label{H1est510}
\frac{\del}{\del \tau} \nm{\N \phi}{L^2}^2 =&\  - 2 \nm{\N \gD \phi}{L^2}^2 -
\int_M \left[ \left< E, \N \phi \otimes \N
\phi \right> + \tr_g E \left( \phi \gD \phi + \frac{1}{2} \brs{\N \phi}^2
\right) \right] dV.
\end{align}
\begin{proof} A direct calculation shows that
\begin{align*}
\frac{\del}{\del \tau} \nm{\N \phi}{L^2}^2 =&\ \int_M \left[ \left<- E, \N \phi
\otimes\N \phi \right> + 2 \left< \N \left( - \gD^2 \phi + \frac{1}{2} \phi
\tr_g E \right), \N \phi \right> - \frac{1}{2}
\tr_g E \brs{\N \phi}^2 \right] dV\\
=&\ - 2 \nm{\N \gD \phi}{L^2}^2 - \int_M \left[ \left< E, \N \phi \otimes \N
\phi \right> + \tr_g E \left( \phi \gD \phi + \frac{1}{2} \brs{\N \phi}^2
\right) \right] dV.
\end{align*}
\end{proof}
\end{lemma}

\noindent We are ready now to derive a fundamental differential inequality for
the evolution of $\mathcal E$.

\begin{prop} \label{glev} There exists a universal constant $C$ such that if
$\ge \leq \frac{1}{2 A}$, one has
\begin{gather} \label{glineq}
\begin{split}
\frac{\del}{\del \tau} \EE(g, \phi) \leq&\ - \frac{\nm{\N \gD
\phi}{L^2}^2}{\nm{\phi}{L^2}^2}\\
&\ + C \left[ \EE^3 + A \nm{E}{L^2} \EE^2 + \left(
A^2 \nm{E}{L^2}^2 + A
\nm{E}{L^2} \right) \left( 1 + \mathcal E \right)\right].
\end{split}
\end{gather}
\begin{proof} Applying Lemmas \ref{Sgrowthlemma5} and \ref{Sgrowthlemma10}
yields
\begin{align*}
\frac{\del}{\del \tau} \mathcal E(\phi_t, g_t) =&\ - 2 \frac{\nm{\N \gD
\phi}{L^2}^2}{\nm{\phi}{L^2}^2} - \frac{ \int_M \left[ \left< E, \N \phi \otimes
\N
\phi \right> + \tr_g E \left( \phi \gD \phi + \frac{1}{2} \brs{\N \phi}^2
\right) \right] dV }{\nm{\phi}{L^2}^2}\\
&\ + 2 \frac{\nm{\N
\phi}{L^2}^2\nm{\gD \phi}{L^2}^2}{\nm{\phi}{L^2}^4} - \frac{\nm{\N \phi}{L^2}^2
\int_M \left[ \frac{1}{2} \tr_g E \phi^2 \right]}{\nm{\phi}{L^2}^4}.
\end{align*}
We need to estimate each of these terms.  First we have
\begin{align*}
\brs{\int_M \left[ \left< - E, \N \phi \otimes \N \phi \right> + \frac{1}{2}
\tr_g E
\brs{\N \phi}^2 \right] dV} \leq&\ C \int_M \brs{E} \brs{\N \phi}^2\\
\leq&\ C \nm{E}{L^2} \nm{\N \phi}{L^4}^2\\
\leq&\ C A \nm{E}{L^2} \left( \nm{\N^2 \phi}{L^2}^2 + \nm{\N \phi}{L^2}^2
\right).
\end{align*}
Integrating by parts and applying the Sobolev inequality yields
\begin{align*}
 \nm{\N^2 \phi}{L^2}^2 \leq&\ \nm{\N \phi}{L^2} \nm{\N \gD \phi}{L^2} + A
\nm{\Rm}{L^2} \left( \nm{\N^2 \phi}{L^2}^2 + \nm{\N \phi}{L^2}^2 \right).
\end{align*}
Choosing $\nm{\Rm}{L^2} \leq \frac{1}{2A}$ yields
\begin{align*}
 \nm{\N^2 \phi}{L^2}^2 \leq&\ C \left( \nm{\N \phi}{L^2} \nm{\N \gD \phi}{L^2} +
\nm{\N \phi}{L^2}^2 \right),
\end{align*}
and hence
\begin{align*}
\brs{\int_M \left[ \left< - E, \N \phi \otimes \N \phi \right> + \frac{1}{2}
\tr_g E
\brs{\N \phi}^2 \right] dV} \leq&\ C A \nm{E}{L^2} \left( \nm{\N \phi}{L^2}
\nm{\N \gD \phi}{L^2} + \nm{\N \phi}{L^2}^2 \right)\\
\leq&\ \frac{1}{3} \nm{\N \gD \phi}{L^2}^2 + C \left( A^2 \nm{E}{L^2}^2 + A
\nm{E}{L^2} \right) \nm{\N \phi}{L^2}^2.
\end{align*}
Next we have
\begin{align*}
\brs{\int_M \tr_g E \phi \gD \phi} \leq&\ \nm{E}{L^2} \nm{\phi}{L^4} \nm{\gD
\phi}{L^4}\\
\leq&\ \nm{E}{L^2} A \left( \nm{\N \phi}{L^2} + \nm{\phi}{L^2} \right) \left(
\nm{\N \gD \phi}{L^2} + \nm{\gD \phi}{L^2} \right)\\
\leq&\ C \nm{E}{L^2} A \left( \nm{\N \phi}{L^2} + \nm{\phi}{L^2} \right) \left(
\nm{\N \gD \phi}{L^2} + \nm{\N \phi}{L^2} \right)\\
\leq&\ \frac{1}{3} \nm{\N \gD \phi}{L^2}^2 + CA \nm{E}{L^2} \left(1 + A
\nm{E}{L^2} \right) \left( \nm{\N \phi}{L^2}^2 + \nm{\phi}{L^2}^2 \right)
\end{align*}
Also, we can estimate
\begin{align*}
 2 \frac{\nm{\N \phi}{L^2}^2 \nm{\gD \phi}{L^2}^2}{\nm{\phi}{L^2}^4} \leq&\ 2
\frac{\nm{\N \phi}{L^2}^{3} \nm{\N \gD \phi}{L^2}}{\nm{\phi}{L^2}^4}\\
=&\ 2 \mathcal E^{\frac{3}{2}} \frac{\nm{\N \gD \phi}{L^2}}{\nm{\phi}{L^2}}\\
\leq&\ \frac{1}{3} \frac{\nm{\N \gD \phi}{L^2}^2}{\nm{\phi}{L^2}^2} + 3 \mathcal
E^3.
\end{align*}
Finally we have
\begin{align*}
 \frac{\nm{\N \phi}{L^2}^2 \int_M \left[ \frac{1}{2} \tr_g E \phi^2
\right]}{\nm{\phi}{L^2}^4} \leq&\ C \frac{\nm{\N \phi}{L^2}^2 \nm{E}{L^2}
\nm{\phi}{L^4}^2}{\nm{\phi}{L^2}^4}\\
\leq&\ C A \frac{\mathcal E \nm{E}{L^2}}{\nm{\phi}{L^2}^2} \left(\nm{\N
\phi}{L^2}^2 + \nm{\phi}{L^2}^2 \right)\\
\leq&\ C A \nm{E}{L^2} \left( \EE^2 + \EE \right).
\end{align*}
Combining these estimates gives the result.
\end{proof}
\end{prop}

\begin{proof}[Proof of Theorem \ref{glestimate}] Let $\phi_T$ denote an
eigenfunction for the first Laplace eigenvalue of $g_T$, and let $\phi_t$, $t
\in [0, T]$ denote the solution to (\ref{phiflow}) with $\phi_T$ as the value at
$t
= T$.  Note that $\int_M \phi_T dV_T = 0$, and by Lemma \ref{phiflowfixing} one
has $\int_M \phi_t dV_t = 0$ for any $0 \leq t \leq T$.    Also, note that
$\phi_t$ does not vanish identically for any $t$.  Indeed, if there was a $t$
such that $\phi_t \equiv 0$, we note that $\phi_s \equiv 0$ is the unique
solution to (\ref{phiflow}), forcing $\phi_T \equiv 0$, a contradiction.  This
implies that for any $t$, $\phi_t$ is a valid test function for estimating
$\gl(g_t)$.

Provided $\bar{\gl}$
is chosen sufficiently small with respect to universal constants, as long as
$\mathcal E(g, \phi) \leq 4 \bar{\gl}$ and $\ge \leq \frac{1}{2 A}$ we conclude
from Proposition \ref{glev} that
\begin{align*}
 \frac{\del}{\del \tau} \EE(g, \phi) \leq&\ C \left(A^2 \nm{E}{L^2}^2 + A
\nm{E}{L^2} + \bar{\gl}^2 \right) \mathcal E + C \left( A^2 \nm{E}{L^2}^2 + A
\nm{E}{L^2} \right)
\end{align*}
Applying Gronwall's inequality yields, for any $t \leq T$, as long as $\sup_{[t,
T]} \EE(g, \phi) \leq 2 \bar{\gl}$,
\begin{align*}
 \EE(g_t, \phi_t) \leq&\ \exp \left[ C \int_t^T \left( A^2 \nm{E}{L^2}^2 + A
\nm{E}{L^2} + \bar{\gl}^2 \right) dt \right] \left( \int_t^T \left( A^2
\nm{E}{L^2}^2 + A \nm{E}{L^2} \right) + \mathcal E(g_T, \phi_T) \right)
\end{align*}
 Next we can estimate
\begin{align*}
 C \int_t^T A^2 \nm{E}{L^2}^2 \leq C A^2 \int_0^T \nm{E}{L^2}^2 \leq C A^2 \ge
\leq \frac{1}{3} \ln 2
\end{align*}
for $\ge$ chosen sufficiently small with respect to $A$ and $\gl(g_T)$. 
Likewise we have
\begin{align*}
 C \int_t^T A \nm{E}{L^2} \leq&\ C A \left( \int_0^T \nm{E}{L^2}^2
\right)^{\frac{1}{2}} \left( \int_0^T \right)^{\frac{1}{2}} \leq C A
\ge^{\frac{1}{2}} T^{\frac{1}{2}} \leq \frac{1}{3} \ln 2
\end{align*}
provided $\ge$ is chosen sufficiently small with respect to $A$ and $T \leq 1$.
Finally we estimate
\begin{align*}
 C \int_t^T \bar{\gl}^2 \leq&\ C T \bar{\gl}^2 \leq \frac{1}{3} \ln 2
\end{align*}
provided $\bar{\gl}$ is chosen sufficiently small with respect to universal
constants and $T \leq 1$.
Combining these estimates yields first of all that for any time $0 \leq t \leq
T$,
\begin{align*}
 \EE(g_t, \phi_t) \leq&\ 2\left(\EE(g_T, \phi_T) + C A^2 \ge^{\frac{1}{2}}
\right) \leq 4 \bar{\gl},
\end{align*}
i.e. the condition $\EE(g_t, \phi_t) \leq 4 \bar{\gl}$ holds on $[0, T]$.  Note
that this last inequality requires that we choose $\ge$ small with respect to
$A$ and $\bar{\gl}$, but of course $\bar{\gl}$ is universal.
Hence
\begin{align*}
 \gl(g_0) =&\ \inf_{\{ \phi | \int_M \phi dV = 0 \}} \frac{\nm{\N
\phi}{L^2}^2}{\nm{\phi}{L^2}^2}
\leq \EE(g_0, \phi_0) \leq 2 \EE(g_T, \phi_T) + C A^2 \ge^{\frac{1}{2}} = 2
\gl(g_T) + C A^2 \ge^{\frac{1}{2}}
\end{align*}
as required.
\end{proof}

\section{Low-Energy Convergence on Three-manifolds} \label{threelowsec}

In this section we prove Theorem \ref{threelow}.  First we recall some
comparison geometry results for manifolds with supercritical $L^p$ bounds on
curvature.
\begin{defn} Let $(M^n, g)$ be a Riemannian manifold, and let $\Rc_-$ denote the
lowest eigenvalue of the Ricci tensor.  Let
\begin{align*}
k(\gl, p) =&\ \int_M \left(\max \{0, (n-1) \gl - \Rc_- \} \right)^p dV.
\end{align*}
\end{defn}

\begin{thm} \label{weakbishgrom} (\cite{PW} Theorem 1.1) \label{petersenvolume}
Let $x \in M$, $\gl \leq 0$, and $p > \frac{n}{2}$ be given, then there is a
constant $C(n, p, \gl, R)$ which is nondecreasing in $R$ such that when $r < R$
we have
\begin{align*}
\left( \frac{\Vol B(x, R)}{v(n, \gl, R)} \right)^{\frac{1}{2p}} - \left(
\frac{\Vol B(x, r)}{v(n, \gl, r)} \right)^{\frac{1}{2p}} \leq&\ C(n, p, \gl, R)
k(\gl, p)^{\frac{1}{2p}},
\end{align*}
where $v(n,\gl,R)$ denotes the volume of a ball of radius $R$ in the simply
connected space form of dimension $n$ with constant sectional curvature
$(n-1)\gl$.
\end{thm}

\begin{thm} (\cite{Gallot} Theorem 3) \label{gallotiso} Let $\ga$ and $D$ be any
positive constants and $p > \frac{n}{2}$.  In any Riemannian manifold $(M^n, g)$
with $\diam(g) \leq D$ satisfying
\begin{align*}
\frac{1}{\Vol(M)} \int_M \left( \max \left\{0, \frac{\Rc_-}{\ga^2(n-1)} - 1
\right\} \right)^p dV \leq&\ \frac{1}{2 \left( e^{B(p) \ga D}  - 1 \right)},
\end{align*}
every domain $\Omega$ satisfies
\begin{align*}
\frac{\Area(\del \Omega)}{\Vol(M)} \geq \gg(\ga, D) \min \left\{
\frac{\Vol(\Omega)}{\Vol(M)}, \frac{\Vol(M \setminus \Omega)}{\Vol(M)}
\right\}^{1 - \frac{1}{p}}.
\end{align*}
\end{thm}

\begin{rmk} \label{volbndrmk} Observe that an $L^p$ energy bound, $p >
\frac{n}{2}$, implies an upper bound on the volume growth of balls using Theorem
\ref{weakbishgrom}.  Therefore in the presence of such a bound the volume is
bounded above in terms of the diameter.
\end{rmk}

\begin{cor} \label{3dimsob} Given $V > 0, D > 0$ there exists $\ge > 0$ so that
if
$(M^3, g)$ is a compact Riemannian manifold with $\Vol(g) \geq V$, $\diam(g)
\leq D, \mathcal F(g) \leq \ge$ then there is a constant $C = C(V,D)$
such that $C_S(g) \leq C$.
\begin{proof} Choose $\ge$ so that
\begin{align*}
\ge \leq \frac{1}{2(e^{B(2) D} - 1)},
\end{align*}
where $B(2)$ is the constant from Theorem \ref{gallotiso}.  Since $\max
\left\{0, \frac{\Rc_-}{n-1} - 1 \right\} \leq \frac{\brs{\Rc}}{n-1}$ and
$\Vol(M) \geq V$, Theorem \ref{gallotiso} applies with $\ga = 1, p = 2$ to
conclude that there is a
constant $\gg$ depending on $V$ and $D$, so that for any subdomain $\Omega$,
\begin{align*}
\frac{\Area(\del \Omega)}{\min \{\Vol(\Omega), \Vol(M \backslash \Omega)
\}^{\frac{2}{3}}} \geq \gg \frac{1}{\min \{\Vol(\Omega), \Vol(M \backslash
\Omega) \}^{\frac{1}{6}}} \geq \gg
\end{align*}
The last line follows since there is a uniform upper bound on the volume of $M$
as observed in Remark \ref{volbndrmk}.  Thus the isoperimetric constant is
bounded, and the result now follows from the discussion in Remark
\ref{isoequiv}.
\end{proof}
\end{cor}

\begin{proof}[Proof of Theorem \ref{threelow}] Note that since we have assumed a
lower bound for the volume, and by Remark \ref{volbndrmk} we have an upper bound
for the volume as long as $\ge \leq 1$, we can rescale to unit volume and it
suffices to show the theorem for such metrics.  First we aim to show a certain
short-time existence statement for solutions to the $L^2$ flow.  We claim that
given $D$ there exists a large constant $K$, and small constants $\ge >
0$ and $T > 0$ such that if $(M^3, g)$ is a compact Riemannian manifold
satisfying
\begin{align} \label{threelow15}
\Vol(g) = 1, \quad \diam(g) \leq D, \quad \mathcal F(g) \leq \ge
\end{align}
 then the solution to the $L^2$ flow exists on $[0, T]$ and
satisfies the estimates
\begin{gather} \label{threelow20}
\begin{split}
\sup_{[0, T]} \diam(g_t) <&\ KD,\\
\sup_{[0, T]} t^{\frac{1}{2}} \brs{\Rm}_{C^0(g_t)}
<&\ 1.
\end{split}
\end{gather}
If the claim were false, then for any choice of $K$, we obtain sequences $\ge_i
\to 0$, $t_i \to 0$, and compact Riemannian manifolds $(M_i^3, g^i)$ such that
$g_i$ satisfies (\ref{threelow15}) with $\mathcal F(g) \leq \ge_i$ and the
solution to the $L^2$ flow with initial condition $g^i$ satisfies the estimates
(\ref{threelow20}) on a maximal time interval $[0, t_i]$.

We aim to derive a contradiction from this statement for sufficiently large $K$.
 First we claim that as long as (\ref{threelow20}) holds there is a uniform
constant $A$ depending on $D$ and $K$ such that
\begin{align*}
\sup_{[0, t_i]} C_S(g^i_t) \leq A
\end{align*}
As long as $\ge_i$ is sufficiently small with respect to $K$ and $D$, this
follows directly from Corollary \ref{3dimsob}.  Suppose now that the second
condition of (\ref{threelow20}) failed at time $t_i$, i.e.
\begin{align*}
\sup_M \brs{\Rm}_{C^0(g^i_{t_i})} = t_i^{-\frac{1}{2}}.
\end{align*}
Define the sequence of time dependent metrics
\begin{align} \label{3flds10}
\til{g}^i(t) = t_i^{-\frac{1}{2}} g^i \left(
t_i \cdot t \right).
\end{align}
The family of metrics $\til{g}^i(t)$ exists on $[0, 1]$ and
\begin{align*}
\sup_{\left[\frac{1}{2}, 1 \right]} C_S(\til{g}^i) \leq&\ A,\\
\sup_{\left[\frac{1}{2}, 1 \right]}
\brs{\til{\Rm}^i}_{C^0(\til{g}^i)} <&\ 2.
\end{align*}
Moreover, by construction $\brs{\til{\Rm}^i}_{C^0(\til{g}^i_{1})} = 1$, and
we let $x^i \in M^i$ be a point realizing this supremum. Using the bound on the
Sobolev constant, one has a lower bound for the volume growth of small balls
(see Lemma \ref{diambnd}), so Cheeger's lemma implies that $\inj_{\til{g}^i}
\geq \nu
> 0$ for some small constant $\nu$. By (\cite{SL21} Theorem 7.1, see also
\cite{SL2LTB} Corollary 1.6) the sequence $\{(M^i, \til{g}^i_t, x^i)\}$ contains
a
subsequence converging to $(M^{\infty}, g^{\infty}_t, x^{\infty})$. Moreover,
one has $\brs{\Rm^{\infty}}_{g^{\infty}_{1}}(x^{\infty}) = 1$. However, since
$\ge_i \to 0$, one
has $\mathcal F(\til{g}^i_{1}) \to 0$. By Fatou's Lemma we can conclude
$\mathcal
F(g^{\infty}_{1}) = 0$, contradicting nonflatness of $(M^{\infty},
g^{\infty},
x^{\infty})$. Thus this possibility is ruled out.

Therefore it must be the case that the first condition of (\ref{threelow20})
fails.  We will work with one element of the sequence and drop the index $i$
from the notation.  We want to derive a contradiction by showing that the first
Laplace eigenvalue of $(M, g(t_i))$ is quite small, then using Theorem
\ref{glestimate} to show that the initial Laplace eigenvalue had to be quite
small, a contradiction.  We will estimate $\gl(g_{t_i})$ using the trick that
\begin{align*}
\gl(M) \leq \max \{ \mu(M_1), \mu(M_2) \}
\end{align*}
where $M_i$ are disjoint open subsets of $M$ and $\mu(M_i)$ denotes the first
Dirichlet eigenvalue of the manifold with boundary.  To that effect, since
$\diam(g_{t_i}) = KD$ we choose
points $x, y$ such that $d_{g_{t_i}}(x, y) = KD$ and estimate
$\mu\left(B_{\frac{KD}{2}}(x) \right)$ above.  To simplify notation let $R =
\frac{KD}{2}$, and let
$\phi \in C^1(M)$ satisfy
\begin{align*}
\phi_{| B_{\frac{R}{2}}(x)} \equiv&\ 1, \qquad \supp \phi \subset B_R(x), \qquad
\brs{\N \phi} \leq \frac{C}{R}.
\end{align*}
Observe that $\phi$ is nonconstant and moreover
\begin{align*}
\int_M \brs{\N \phi}^2 \leq&\ \frac{C}{R^2} \Vol(B_R(x_i)), \qquad \int_M
\phi^2 \geq \Vol(B_{\frac{R}{2}}(x_i)).
\end{align*}
However, since the Sobolev constant is bounded, there is a certain constant
$\eta
= \eta(K, D)$ such that
\begin{align*}
\Vol(B_R(x)) \geq \eta(K, D).
\end{align*}
Applying Theorem \ref{weakbishgrom} with $r = \frac{R}{2}$ and $\gl = 0$, we
observe that
\begin{align*}
\Vol(B_{\frac{R}{2}}(x)) \geq C \left( \Vol(B_R(x))^{\frac{1}{4}} - C(n, R)
k(\gl,
p)^{\frac{1}{4}} \right)^4.
\end{align*}
If we choose $\ge$ small enough so that
\begin{align*}
C(n, R) k(\gl, p) \leq \frac{1}{2} \eta(K,D)
\end{align*}
we may conclude that
\begin{align*}
\Vol(B_{\frac{R}{2}}(x)) \geq C \Vol(B_{R}(x))
\end{align*}
for a universal constant $C$.  We conclude that for a new constant $C$ one
has
\begin{align*}
\mu \left( B_{R}(x) \right) \leq \frac{\nm{\N
\phi}{L^2}^2}{\nm{\phi}{L^2}^2} \leq \frac{C}{2 K^2 D^2}.
\end{align*}
One estimates the first Dirichlet eigenvalue of the ball of radius $R$
around $y$ identically, and thus we yield
\begin{align*}
\gl(g_{t_i}) \leq \frac{C}{2 K^2D^2}.
\end{align*}
Now suppose that $K$ is sufficiently large that $\frac{C}{K^2 D^2} \leq
\bar{\gl}$, where $\bar{\gl}$ is the constant from Theorem \ref{glestimate}. 
Also note that since at each time the metric has bounded volume and bounded
Sobolev
constant $H_1^2 \to L^6$, a simple application of H\"older's inequality shows
that the constant $A$ of Theorem \ref{glestimate} is bounded in terms of the
given Sobolev constant.  We now choose $\ge$ sufficiently small with respect to
the this bound (which depends on $K$ and $D$), so that Theorem \ref{glestimate}
applies to conclude
\begin{align*}
 \gl(g_0) \leq 2 \gl(g_{t_i}) + C A ^2 \ge^{\frac{1}{2}} \leq \frac{C}{K^2D^2}.
\end{align*}
However, from Theorem \ref{gallotiso} we know that if $\ge$ is chosen
sufficiently small there is a lower
bound on the isoperimetric ratio
\begin{align*}
 \frac{\Area(\del \Omega)}{\min \{\Vol(\Omega), \Vol(M \backslash \Omega) \}}
\geq \gg(D) \Vol(M)^{\frac{1}{2}} \min \{\Vol(\Omega), \Vol(M \backslash \Omega)
\}^{-\frac{1}{2}}.
\end{align*}
Since $\Vol(g_0) = 1$, we conclude
\begin{align*}
 h(M, g_0) := \inf_{\Omega \subset M} \frac{\Area(\del \Omega)}{\min
\{\Vol(\Omega), \Vol(M \backslash \Omega)\}} \geq&\ \gg(D).
\end{align*}
By Cheeger's inequality \cite{Cheeger} we conclude
\begin{align*}
 \gl(g_0) \geq \frac{h(M, g_0)^2}{4} \geq \frac{\gg(D)^2}{4}.
\end{align*}
Choosing $K$ sufficiently large with respect to $\gg(D)$, we may conclude
\begin{align*}
 \frac{\gg(D)^2}{4} \leq \gl(g_0) \leq \frac{C}{K^2 D^2} < \frac{\gg(D)^2}{4},
\end{align*}
a contradiction.  Thus the claim of uniform short time existence follows.

To finish the proof we use a version of the implicit function theorem for
solutions to the $L^2$ flow near flat manifolds.  In particular, we continue
arguing by contradiction, and given $D > 0$ we assume that for every $\ge > 0$
there is a three-manifold satisfying the hypotheses of the theorem but for which
the flow does not exist for all time and converge to a flat metric.  Choose a
sequence $\ge_i \to 0$ and $(M_i^3, g^i)$ realizing this possibility.  By the
discussion of the diameter bound above and Theorem \ref{globalsmoothing}, for
sufficiently small $\ge_i$ we have that the solution to
the $L^2$ flow exists on $[0, T]$, and moreover
\begin{align*}
\diam(g^i_{T}) <&\ KD,\\
\brs{\N^k \Rm}_{C^0(g^i_{T})} \leq&\ C_k.
\end{align*}
By the discussion above, we also conclude a uniform lower bound on the
injectivity radius of $g^i_{T}$.  It follows from \cite{Hamilton} Theorem 2.3
that we may take a subsequence of $\{(M^3_i, g^i_{T})\}$ which converges in
the $C^k$ topology for any $k$, necessarily to a flat metric.  At this point one
can repeat the argument of (\cite{SL22} Theorem 1.6) to conclude that for
$g^i_{T}$ sufficiently close to a flat metric in $C^k$, the $L^2$ flow exists
for all time and converges exponentially to a flat metric.  Given this
exponential convergence, it is a straightforward matter to show that the volume
normalized $L^2$ flow also exists for all time and converges to a flat metric.
\end{proof}

\begin{cor} Given $V > 0, D > 0$ there exists $\ge > 0$ sufficiently small so
that the
space of metrics on $T^3$ satisfying $\Vol(g) \geq V$, $\diam(g) \leq D$,
$\mathcal
F(g) \leq \ge$ is connected in the $C^{\infty}$ topology.
\begin{proof} Theorem \ref{threelow} guarantees that for $\ge > 0$ sufficiently
small, metrics on $T^3$ satisfying $\Vol(g) \geq V$, $\diam(g) \leq D$ and
$\mathcal F(g) \leq \ge$ are smoothly deformable to flat metrics.  Since the
space of flat metrics on $T^3$ is path-connected, the corollary follows.
\end{proof}
\end{cor}

\section{Low Energy Convergence on Four-manifolds} \label{fourmnflds}

In this section we investigate the $L^2$ flow with low energy on four-manifolds.
 The optimal convergence result in this direction would be an analogue of
Theorem \ref{threelow}, i.e. given energy sufficiently small with respect to the
Sobolev constant, the solution to the $L^2$ flow exists for all time and
converges to a flat metric.  The first test of this claim is to determine if
there are any other critical points of $\mathcal F$ in this regime.

\begin{prop} Given $A > 0$, there exists $\ge > 0$ so that if $(M^4, g)$ is a
compact Riemannian manifold satisfying
\begin{align*}
 \grad \FF \equiv&\ 0\\
 C_S(g) \leq&\ A\\
 \mathcal F(g) \leq&\ \ge,
\end{align*}
 then $g$ is flat.
\begin{proof} If the statement was false, then given $A > 0$, there exists a
sequence $\ge_i \to 0$ and a sequence of compact Riemannian manifolds $\{(M^4_i,
g^i) \}$ of compact critical, nonflat four-manifolds satisfying the hypotheses
of the theorem.  By rescaling, we may assume without loss of generality that the
metrics satisfy $\Vol(g^i) = 1$.  We first claim is that there is a uniform
curvature bound along the sequence.  If not, there is some subsequence such that
\begin{align*}
K_i := \brs{\Rm}_{g^i}(x_i) = \brs{\Rm}_{C^0(g^i)} \to \infty
\end{align*}
Observe that since the metrics are fixed points of the $L^2$ flow, the sequence
of manifolds $\{(M^4_i, K_i g^i, x_i) \}$ has uniform bounds on all
covariant derivatives of curvature by Theorem \ref{globalsmoothing}, and so the
sequence converges to a noncompact, nonflat
critical four-manifold.  But since $\ge_i \to 0$ it follows that this limiting
manifold must be flat, a contradiction.

Since there is a uniform bound on the curvature along the sequence, there are
also uniform bounds on all higher derivatives of curvature.  Also, since the
Sobolev constants are bounded, we obtain a uniform lower bound on the
injectivity radius of $g^i$, and a uniform upper bound on the diameter.  It
follows that we may take a limit of $\{(M^4_i, g^i) \}$, which is necessarily
flat.  In particular, for large enough $i$ $g^i$ is $C^k$-close to a flat metric
for arbitrary $k$.  It follows from (\cite{SL22} Theorem 1.6) that the solution
to the $L^2$ flow with initial condition $g^i$ exists for all time and converges
to a flat metric.  But since $g^i$ is critical, the flow is stationary,
therefore $g^i$ is already flat, a contradiction.  The proposition follows.
\end{proof}
\end{prop}

Next we give the proof of Theorem \ref{fourlow1}, which says that, in
determining the behavior of the $L^2$ flow on four-manifolds with energy small
with respect to the Sobolev constant, the problems lie in understanding the
short-time behavior of the Sobolev constant.

\begin{proof}[Proof of Theorem \ref{fourlow1}] We begin by showing a certain
uniform short time existence
statement. In particular, we claim that we may choose $\ge(A, B)$ so that if
$(M^4, g_0)$ is as in the statement of the theorem, then there exists a uniform
$T(A,B) > 0$ so that the
solution to the $L^2$ flow exists on $[0, T]$ and moreover satisfies the
estimates
\begin{gather} \label{continuityestimates}
\begin{split}
\sup_{t \in [0, T]} C_S(g_t) <&\ 2 B,\\
t^{\frac{1}{2}} \brs{\Rm}_{C^0(g_t)} <&\ 1.
\end{split}
\end{gather}

If this is false, then we have a sequence $\ge_i \to 0$ and a
sequence of compact Riemannian four-manifolds $\{(M^i, g^i)\}$ satisfying
$C_S(g^i) \leq
B$ and $\FF(g^i) \leq \ge_i$, such that if $g^i_t$ denotes the solution to
the $L^2$ flow with initial condition $g^i$, one the estimates of
(\ref{continuityestimates}) fails at some time $t_i < 1$. Suppose there existed
a subsequence where the second
condition of (\ref{continuityestimates}) failed at $t_i$, i.e.
$t_i^{\frac{1}{2}} \brs{\Rm}_{C^0(g^i_{t_i})} = 1$. Define
\begin{align*}
\til{g}^i(t) = t_i^{-\frac{1}{2}} g^i \left( t_i \cdot t \right).
\end{align*}
Each one-parameter family $\til{g}^i(t)$ exists on $[0, 1]$ and moreover
\begin{align*}
\sup_{[0, 1]} C_S(\til{g}^i) < 2 B, \quad \sup_{\left[\frac{1}{2}, 1 \right]}
\brs{\til{\Rm}^i}_{C^0(\til{g}^i)} < 2.
\end{align*}
Moreover, by construction $\brs{\til{\Rm}^i}_{C^0(\til{g}^i_1)} = 1$, and we let
$x^i$ be a point realizing this supremum. Using the bound on $C_S(\til{g}^i)$,
one automatically obtains a scale-invariant lower bound on the volume growth of
balls, and then it follows from Cheeger's lemma that $\inj_{\til{g}^i} \geq
\nu > 0$ for some small constant $\nu$. By (\cite{SL21} Theorem 7.1, see also
\cite{SL2LTB} Corollary 1.6) the sequence $(M^i, \til{g}^i_t, x^i)$ contains a
subsequence converging to $(M^{\infty}, g^{\infty}_t, x^{\infty})$. Moreover,
one has $\brs{\Rm^{\infty}}(x^{\infty}) = 1$. However, since $\ge_i \to 0$, one
has $\mathcal F(\til{g}^i_1) \to 0$. By Fatou's Lemma we can conclude $\mathcal
F(g^{\infty}_1) = 0$, contradicting nonflatness of $(M^{\infty}, g^{\infty},
x^{\infty})$. Thus this possibility is ruled out.

Now using the curvature decay we will show that the first condition of
(\ref{continuityestimates}) holds on $[0, T]$ for $\ge$ sufficiently small. In
particular, we first note from Theorem \ref{globalsmoothing} that on $[0,
T]$ we may conclude uniform estimates
\begin{align} \label{decay}
t^{\frac{k+2}{4}} \brs{\N^k \Rm} < C_k.
\end{align}
Furthermore, applying Proposition \ref{Egrowthprop} we may choose $\ge$ small
with respect to $B$ such that, as long as
(\ref{continuityestimates}) holds, we have
\begin{align*}
\sup_{[0, T]} \nm{\grad \FF}{L^2}^2 \leq 3 A^2.
\end{align*}
Now applying Theorem \ref{multsob} with $m = 2$, $q = 6$, $p = \infty$ and $\ga
= \frac{3}{4}$, we conclude that, as long as (\ref{continuityestimates}) holds,
\begin{align} \label{intest}
\int_0^T \nm{\grad \FF}{\infty} \leq C B^{\frac{3}{4}} A^{\frac{1}{4}} \int_0^T
\left(\nm{\N \grad \FF}{L^q}^{\frac{3}{4}} + \nm{\grad \FF}{L^q}^{\frac{3}{4}}
\right) dt.
\end{align}
Using (\ref{decay}) and the fact that $\Vol(g) = 1$, we conclude that
\begin{align*}
\nm{\N \grad \FF}{L^q} \leq&\ C \Vol^{\frac{1}{q}} \left(\nm{\N^3 \Rm}{\infty} +
\nm{\N \Rm}{\infty}\nm{\Rm}{\infty} \right)\\
\leq&\ C t^{-\frac{5}{4}}.
\end{align*}
Similarly
\begin{align*}
\nm{\grad \FF}{L^q} \leq&\ C t^{-1}.
\end{align*}
Plugging these into (\ref{intest}) yields, as long as $T \leq 1$,
\begin{align*}
\int_0^T \nm{\grad \FF}{\infty} \leq C B^{\frac{3}{8}} A^{\frac{1}{4}}
T^{\frac{1}{16}}.
\end{align*}
In particular, given $\gd > 0$ we may choose $T$ sufficiently small with respect
to $A$ and $B$ so that, for all $t \in [0, T]$,
\begin{align*}
(1 + \gd)^{-1} g_0 \leq g_t \leq (1 + \gd) g.
\end{align*}
Furthermore for $\gd$ chosen sufficiently small with respect to universal
constants this implies 
\begin{align*}
C_S(g_t) \leq \frac{3}{2} C_S(g_0),
\end{align*}
and the short time existence claim is finished.

To finish the proof we apply an analytic stability result for the $L^2$ flow
near
flat metrics. It follows from Theorem \ref{globalsmoothing} that for each $k$
one has
uniform estimates on $\brs{\N^k \Rm}_{C^0(g^i_{1})}$. Furthermore, since
$\Vol(g^i_{1}) = 1$, from Lemma \ref{diambnd} we have a uniform upper bound on
$\diam(g^i_{1})$, and from (\ref{volbnd}) and Cheeger's lemma a uniform lower
bound on $\inj(g^i_{1})$. It follows from (\cite{Hamilton} Theorem 2.3) that
there exists a subsequence of $\{g^i_{1} \}$ converging in any $C^k$ norm to a
flat metric. It follows from (\cite{SL22} Theorem 1.6) that for
sufficiently large $i$, the solution to $L^2$ flow with initial condition
$g^i_{1}$ exists for all time and converges exponentially to a flat metric. This
finishes the proof.
\end{proof}

\section{Higher Dimensions} \label{highdim}

We first observe a simple proposition which exemplifies the role of the
dimension in understanding solutions to the $L^2$ flow. In particular, we note
that finite time singularities certainly occur in dimensions $n \geq 5$.

\begin{prop} Consider $(S^n, g_{S^n})$ where $g_{S^n}$ is the metric of constant
sectional curvature $K \equiv 1$. The solution to (\ref{flow}) with initial
condition $g_{S^n}$ exists
\begin{itemize}
\item{on $[0, \infty)$ and satisfies $g(t) = \sqrt{1 + c_n t} g_{S^n}$ for $n
= 2, 3$.}
\item{on $[0, \infty)$ and satisfies $g(t) = g_{S^n}$.}
\item{on $[0, \frac{1}{c_n})$ and satisfies $g(t) = \sqrt{1 - c_n t} g_{S^n}$,}
\end{itemize}
where $c_n$ is a constant depending on the dimension.
\begin{proof} The metric $g_{S^n}$ satisfies
\begin{align*}
\N \Rc_{g_{S^n}} = 0, \qquad \check{R}_{g_{S^n}} =&\ \frac{2}{n} \brs{\Rm}^2 g =
\frac{2}{n} \left(n(n-1) \right) g
\end{align*}
It follows that the solution to the $L^2$ flow with initial condition $A
g_{S^n}$
reduces to the ODE
\begin{align*}
\dt A =&\ \frac{\left(\frac{1}{n} - \frac{1}{4} \right) 2n(n-1)}{A}.
\end{align*}
The proposition follows immediately.
\end{proof}
\end{prop}

As it turns out, not only does the $L^2$ flow encounter finite time
singularities in dimension $n\geq 5$, in general they need not satisfy a
noncollapsing estimate.  We next recall Perelman's no local collapsing result
for Ricci flow. First we recall the
definition of $\gk$-collapsing on a given scale.

\begin{defn} A Riemannian manifold $(M^n, g)$ is said to be
\emph{$\gk$-collapsed at the scale $r$} if there exists $x \in M$ such that
$\brs{\Rm} \leq r^{-2}$ for all points in $B(x, r)$, and
\begin{align*}
\frac{\Vol(B(x, r)}{r^n} \leq \gk.
\end{align*}
\end{defn}

\begin{thm} (\cite{P1}) Let $g(t), t \in [0, T)$ be a smooth solution to the
Ricci flow on a closed manifold $M^n$. If $T < \infty$, then for any $\rho \in
(0, \infty)$ there exists $\gk = \gk(g(0), T, \rho)$ such that $g(t)$ is
$\gk$-noncollapsed below the scale $\rho$ for all $t \in [0, T)$.
\end{thm}

\noindent This theorem has a corollary fundamental to the analysis of finite
time singularities of Ricci flow.

\begin{cor} (\cite{P1}) Let $(M^n, g(t)), t \in [0, T), T < \infty$ be a
solution to the
Ricci flow on a closed manifold. For every $C > 0$ there exists $\ga > 0$
depending on $C, g(0)$, and $T$ such that if $(x, t)$ satisfies
\begin{align*}
\brs{\Rm}_{g(t)} \leq C K
\end{align*}
on $B_{g(t)} \left(x, \frac{1}{\sqrt{CK}} \right)$, where $K = \brs{\Rm}_{g(x,
t)}$, then
\begin{align*}
\inj_{g(t)}(x) \geq \frac{\ga}{\sqrt{K}}.
\end{align*}
\end{cor}

Alas, in high dimensions it is possible for solutions to the $L^2$ flow to fail
to
satisfy an injectivity radius estimate on the scale of maximum curvature.

\begin{prop} Consider $M^6 = S^5 \times S^1$, and let $g_0 = A_0 g_{S^5} \oplus
B_0 g_{S^1}$, where $g_{S^n}$ denotes the metric of constant sectional curvature
$K \equiv 1$. The solution to the $L^2$ flow with this initial condition exists
on
a finite time interval $\left[0, T\right]$, and moreover,
\begin{align*}
\lim_{t \to T} \brs{\Rm}_{g_t} \inj_{g_t}^2 = 0.
\end{align*}
\begin{proof} By the uniqueness of solutions to the $L^2$ flow, the isometry
group
of $g_0$ is preserved along the flow. In particular, the flow will reduce to an
ODE on the parameters $A$ and $B$, i.e. we may express
\begin{align*}
g_t = A_t g_{S^5} \oplus B_t g_{S^1}.
\end{align*}
The curvature tensor of any such metric is parallel, therefore $\gd d \Rc \equiv
0$ along the flow. Next define the dimensional constant $c_n :=
\brs{\Rm(g_{S^n})}_{g_{S^n}}^2$. It follows that
\begin{align*}
\brs{\Rm}^2 g =&\ \frac{c_5}{A^2} \left( A g_{S^5} \oplus B g_{S^1} \right)
\end{align*}
Next, since $g_{S^5}$ has constant curvature and $g_{S^1}$ is flat, one can
check that $\check{R}$ must be a multiple of $g_{S^n}$. Using that $\tr_g
\check{R} = \brs{\Rm}^2$, it follows that
\begin{align*}
\check{R} =&\ \frac{c_5}{5 A} g_{S^5}.
\end{align*}
It follows that the solution to the $L^2$ flow is reduced to the system of ODEs
\begin{align*}
\dt A =&\ \frac{c_5}{5A} - \frac{c_5}{4 A} = - \frac{c_5}{20 A}\\
\dt B =&\ - \frac{c_5 B}{4 A^2}.
\end{align*}
The solution exists on $ \left[0, \frac{10 A_0^2}{c_5} \right]$, and one has
\begin{align*}
A_t = \sqrt{ A_0^2 - \frac{t c_5}{10}}.
\end{align*}
We have shown the existence statement, next we show that $\brs{\Rm} \inj^2$
approaches zero at the singular time. One directly computes that
\begin{align*}
\dt \ln \left( \frac{B}{A^p} \right) =&\ \dt \left( \ln B - p \ln A \right)\\
=&\ \frac{\dt B}{B} - p \frac{\dt A}{A}\\
=&\ - \frac{c_5}{4 A^2} - p \left( - \frac{c_5}{20 A^2} \right)\\
=&\ \frac{c_5}{4 A^2} \left( \frac{p}{5} - 1 \right).
\end{align*}
Therefore we have that
\begin{align*}
\frac{B_t}{A_t^5} = \frac{B_0}{A_0^5}.
\end{align*}
Now let $\theta$ denote the standard coordinate on $S^1$. For any $x \in S^5$,
since the metric is a Riemannian product, the lateral curve $\gg(\theta) = (x,
\theta)$ is a geodesic. Its length is $2 \pi \sqrt{B}$, and is not minimizing
past length $\pi \sqrt{B}$. Thus $\inj_{g_t} \leq \pi \sqrt{B_t}$. Let $T =
\frac{10 A_0^2}{c_5}$. Since $\brs{\Rm}^2_{g_t} = \frac{c_5}{A^2}$ we thus have
that
\begin{align*}
\lim_{t \to T} \brs{\Rm}_{g_t} \inj_{g_t}^2 \leq&\ \lim_{t \to T}
\frac{c_5}{A_t} \pi B_t = \lim_{t \to T} \frac{\pi c_5 B_0}{A_0^5} A_t^4 = 0.
\end{align*}
\end{proof}
\end{prop}

\begin{rmk} Note that this example shows that this behavior can occur in any
dimension $n \geq 6$, by simply taking a product with a torus of arbitrary
dimension. In all likelihood one can find an example in dimension $n
= 5$ which experiences collapse at the scale of maximum curvature, but so far no
easy example presents itself. Therefore, any effort to show that finite time
singularities of the $L^2$ flow are not collapsed at the scale of curvature must
take the dimension into account in some identifiable way.
\end{rmk}

\section{Conjectural Framework} \label{conjsec}

The gradient for of the functional $\mathcal F(g)$ can be thought of as an
intrinsic Riemannian analogue of the Yang-Mills energy.  Observing the scaling
law $\mathcal F(\gl g) = \gl^{\frac{n}{2} - 2} \mathcal F(g)$, one can hope for
good regularity properties of the gradient flow in dimensions $n = 2, 3$, and
dimension $n = 4$ with sufficiently small energy.  Let us recall some results
from the theory of Yang-Mills flow which illustrate this behavior.

\begin{thm} (Rade \cite{Rade}) Let $(M^n, g)$ be a compact Riemannian manifold
with $n = 2, 3$.  Let $E \to M$ denote the total space of a vector bundle over
$M$ with semisimple structure group.  If $A_0$ denotes a connection on $E$, the
solution to the Yang-Mills flow
with initial condition $A_0$ exists for all time and converges to a Yang-Mills
connection.
\end{thm}

\begin{rmk} The proof is via Moser iteration, where the supercriticality of the
functional exhibits itself in a clear fashion. An a-priori bound on the Sobolev
constant of the base manifold is essential to this proof, and as we have
remarked above it is precisely this lack of a-priori control over the Sobolev
constant which provides such extreme difficulty in understanding solutions to
the $L^2$ flow. Furthermore, the issue of convergence is not immediately settled
by this proof as the estimates degenerate at infinite time.
\end{rmk}

\begin{thm}\label{struwe} (Struwe \cite{Struwe}) Let $(M^4, g)$ be a compact
Riemannian manifold, and let $E \to M$ denote the total space of a vector bundle
over $M$ with semisimple structure group.  Let $A_0$ denote a connection on $E$.
The solution to the Yang-Mills flow with initial condition $A_0$ exists on a
maximal time interval $[0, T)$, and
\begin{align*}
T = \sup \left\{ \bar{t} > 0 | \exists R > 0, \sup_{x_0 \in M, 0 \leq t \leq
\bar{t}} \left( \int_{B_R(x_0)} \brs{F(t)}^2 dV \right) < \ge_0 \right\} 
\end{align*}
where $\ge_0 = \ge_0(E) > 0$.
\end{thm}

\begin{rmk} Struwe proves more than this, and one should consult \cite{Struwe}
for the precise result.  Observe that one consequence is that if the initial
global energy is sufficiently small the flow will exist for all time.
\end{rmk}

\noindent With these results as guideposts, we can make a  natural conjecture:

\begin{conj} \label{mainconj} (Main existence conjecture): Let $(M^n, g)$ be a
compact Riemannian
manifold and suppose either
\begin{itemize}
\item{ $n = 2, 3$, or}
\item{ $n = 4$ and $\nm{\Rm}{L^2} \leq \ge$,}
\end{itemize}
where $\ge$ is some universal constant.  Then the solution to the $L^2$ flow
exists for all time.
\end{conj}

\begin{rmk} Certainly one cannot expect convergence at $t = \infty$ for $n = 3,
4$, as
solutions in general will collapse.
\end{rmk}

\begin{rmk} The case $n = 2$ of Conjecture \ref{mainconj} was established in
\cite{SL2Surfaces}.  While it is natural to expect convergence of the
flow to a constant scalar curvature metric, this is not yet known in general.
\end{rmk}

Observe that the $n = 4$ conjecture is actually stronger than the directly
analogous statement of Theorem \ref{struwe}.  In particular, we have asked that
the constant $\ge$ be independent of the underlying topology, let alone the
initial metric.  With this in mind, a certain weaker conjecture when $n = 4$ may
be more attainable.

\begin{conj} Given $C > 0$ there exists $\ge(C) > 0$ so that if $(M^4, g)$ is a
compact
Riemannian manifold with $C_S(g) \leq C$ and $\mathcal F(g) \leq \ge$, the
solution to  the $L^2$ flow exists for all time and converges to a flat metric.
\end{conj}

\noindent One cannot help but wonder if these conjectures provide a path towards
resolving
an old question of Gromov:

\begin{conj} (Gromov) There exists $\ge > 0$ so that if $(M^4, g)$
satisfies $\nm{\Rm}{L^2} \leq \ge$, then $M$ admits an $\mathcal F$-structure.
\end{conj}

\bibliographystyle{hamsplain}

\end{document}